\documentclass[11pt, leqno]{amsart}
\usepackage{graphicx}
\usepackage{amsfonts,delarray,amssymb,amsmath,amsthm,a4,a4wide}
\usepackage{latexsym}
\usepackage{epsfig}
\usepackage{color}
\usepackage[margin=1in]{geometry} 
\newtheorem{thm}{Theorem}[section]

\newtheorem{prop}[thm]{Proposition}
\theoremstyle{definition}

\newtheorem{rem}[thm]{Remark}
\numberwithin{equation}{section}

\renewcommand{\div}{\mbox{div}\,}
\newcommand{\trace}{\mbox{trace}\,}

\newcommand{\R}{\mathbb R}

\newcommand{\e}{\varepsilon}

\newcommand{\p}{\partial}

\newcommand{\comment}[1]{}

\newenvironment{myindentpar}[1]%
{\begin{list}{}%
         {\setlength{\leftmargin}{#1}}%
         \item[]%
}
{\end{list}}

\begin{document}

\title[Singular Abreu equations in higher dimensions ]{On Singular Abreu equations in higher dimensions}
\author{Nam Q. Le}
\address{Department of Mathematics, Indiana University,
Bloomington, 831 E 3rd St, IN 47405, USA.}
\email{nqle@indiana.edu}

\thanks{The research of the author was supported in part by the National Science Foundation under grant DMS-1764248}

\begin{abstract}
We study the solvability of the second boundary value problem of a class of highly singular, fully nonlinear fourth order equations of Abreu type in higher dimensions under either a smallness condition or radial symmetry.

\end{abstract}
\maketitle
\section{Introduction and statements of the main results }
In this paper, which is a sequel to \cite{Le18}, we study the solvability of the second boundary value problem of a class of highly singular, fully nonlinear fourth order equations of Abreu type for a uniformly convex function $u$:
\begin{equation}
\label{Abreu}
  \left\{ 
  \begin{alignedat}{2}\sum_{i, j=1}^{n}U^{ij}w_{ij}~& =F(\cdot, u, Du, D^2 u)~&&\text{in} ~\Omega, \\\
 w~&= (\det D^2 u)^{-1}~&&\text{in}~ \Omega,\\\
u ~&=\varphi~&&\text{on}~\p \Omega,\\\
w ~&= \psi~&&\text{on}~\p \Omega.
\end{alignedat}
\right.
\end{equation}
Here and throughout, $U=(U^{ij})_{1\leq i, j\leq n}$ is the cofactor matrix of the Hessian matrix $D^2 u=(u_{ij})_{1\leq i, j\leq n}\equiv \left(\frac{\p^2 u}{\p x_i \p x_j}\right)_{1\leq i, j\leq n}$;
$\varphi\in C^{3, 1}(\overline{\Omega})$, $\psi\in C^{1, 1}(\overline{\Omega})$ with $\inf_{\p\Omega}\psi>0$. The left hand side of (\ref{Abreu}) usually appears in Abreu's equation \cite{Ab} in the problem of finding K\"ahler metrics of
constant scalar curvature in complex geometry.
\vglue 0.2cm
This type of equation arises from studying 
approximation of convex functionals such as the Rochet-Chon\'e model in product line design \cite{RC} whose Lagrangians depend on the gradient variable, subject to a convexity constraint.
Carlier-Radice \cite{CR} studied equation of the type (\ref{Abreu}) when $F$ does not depend on the Hessian variable. 
 When the function $F$ depends on the Hessian variable, (\ref{Abreu})
 was studied in \cite{Le18} in two dimensions, including the case $F=-\Delta u$. 
 
 \vglue 0.2cm
 Note that (\ref{Abreu}) consists of a Monge-Amp\`ere equation for $u$ in the form of $\det D^2 u=
w^{-1}$ and a linearized Monge-Amp\`ere equation for $w$ in the form of $$\sum_{i, j=1}^{n} U^{ij} w_{ij}=F(\cdot, u, Du, D^2 u)$$ because
the coefficient matrix $(U^{ij})$ comes from linearization of the Monge-Amp\`ere operator: $$U^{ij}=\frac{\p\det D^2 u}{\p u_{ij}}.$$ 
\vglue 0.2cm
The solvability of second boundary problems such as (\ref{Abreu}) is usually established via a priori fourth order derivative estimates and degree theory.  Two of the key ingredients for the a priori estimates are to establish (see \cite{Le18}): 
\begin{myindentpar}{1cm}
(i) positive lower and upper bounds for the Hessian determinant $\det D^2 u$; and\\
 (ii) global H\"older continuity for $w$ from 
global H\"older continuity of the linearized Monge-Amp\`ere equation with right hand side having low integrability.
\end{myindentpar}
By Theorem 1.7 in combination with Lemma 1.5 in 
\cite{LN}, any integrability more than $n/2$ right hand side of the linearized Monge-Amp\`ere equation suffices for the global H\"older continuity and $n/2$ is the precise threshold. The reason to restrict the analysis in \cite{Le18} to two dimensions even for the simple case $F=-\Delta u$ is that either $\Delta u$ is just a measure or it belongs to 
$\Delta u\in L^{1+\e_0}(\Omega)$ where 
$\e_0>0$ can be arbitrary small. The condition $n/2<1+\e_0$ with small $\e_0$ naturally leads to $n=2$.  
\vglue 0.2cm
In all dimensions, once we have the global H\"older continuity of $w$ together with the lower and upper bounds on $\det D^2 u$,
we can apply the global $C^{2,\alpha}$ estimates for the Monge-Amp\`ere equation in \cite{S, TW08} 
 to conclude that $u\in C^{2,\alpha}(\overline{\Omega})$.  We update this information to $U^{ij} w_{ij}=F(\cdot, u, Du, D^2 u)$ to have a second order uniformly elliptic  equation for $w$ with global H\"older continuous coefficients
 and bounded right hand side. This gives second order derivatives estimates for $w$. Now, fourth order derivative estimates for $u$ easily follows.  
 \vglue 0.2cm
  In this paper, we consider the higher dimensional case of (\ref{Abreu}), focusing on the right-hand side being of $p$-Laplacian type. In this case, the first two equations of (\ref{Abreu}) arise as the Euler-Lagrange equation of the convex functional
  \begin{equation}
  \label{Jfn}
  J_p(u):=\int_{\Omega} \left(\frac{|Du|^p}{p} - \log \det D^2 u\right) dx.
  \end{equation}
  When $p=2$, that is, (\ref{Abreu}) with $F=-\Delta u$, the a priori lower bound on $\det D^2 u$ in \cite{Le18} breaks down when $n\geq 3$. 
  
  Key to this analysis in \cite{Le18} is the fact that $\trace (U)= \Delta u$ in dimensions n=2.
  With this crucial fact, one can use
  $$U^{ij} (w+ \frac{1}{2} |x|^2)_{ij}=-\Delta u + \trace (U) \geq 0$$
  and then applying the maximum principle to conclude that $w+ \frac{1}{2}|x|^2$ attains its maximum on $\p\Omega$ from which the upper bound on $w$ follows which in turn implies the desired lower bound on $\det D^2 u$.
  
  If $n\geq 3$, the ratio $\frac{\trace(U)}{\Delta u}$ can be in general as small as we want; in fact, this is the case, say, when one eigenvalue of $D^2 u$ is $1$  while all other $n-1$ eigenvalues are a small constant. 
 
  Here, we use a new technique to solve (\ref{Abreu}) when $F=-\gamma \div (|Du|^{p-2} Du)$ where $p\geq 2$ and $\gamma$ is small. 
  More generally, our main result states as follows.
 \begin{thm}
  \label{thm_sm}
  Assume $n\geq 3$. 
   Let $\Omega$ be an open, smooth, bounded and uniformly convex domain in $\R^n$.
  Let $\psi\in C^{2, \beta}(\overline{\Omega})$ with $\inf_{\p\Omega}\psi >0$ and let $\varphi \in C^{4, \beta}(\overline{\Omega})$ where $\beta\in (0,1)$. Let $F(\cdot, z, {\bf p}, {\bf r}): \overline{\Omega} \times \R\times \R^n\times \R^{n\times n}$ 
  be a smooth function such that: 
  \begin{myindentpar}{1cm}
  (i) it maps compact subsets of $\overline{\Omega} \times \R\times \R^n\times \R^{n\times n}$ into compact subsets of $\R$ and \\
  (ii) 
  $F(x, u(x), Du(x), D^2 u(x))\leq 0$ in $\Omega$ for all $C^2$ convex function $u$.
  \end{myindentpar}
 If  $\gamma>0$ is a small constant depending only on $\beta, \varphi, \psi, n, F$ and $\Omega$, then there is a uniform convex solution $u\in C^{4,\beta}(\overline{\Omega})$ to the following
 second boundary value problem:
\begin{equation}
\label{Abreu_small}
  \left\{ 
  \begin{alignedat}{2}\sum_{i, j=1}^{n} U^{ij}w_{ij}~& =\gamma F(\cdot, u, Du, D^2 u)~&&\text{in} ~\Omega, \\\
 w~&= (\det D^2 u)^{-1}~&&\text{in}~ \Omega,\\\
u ~&=\varphi~&&\text{on}~\p \Omega,\\\
w ~&=  \psi~&&\text{on}~\p \Omega.
\end{alignedat}
\right.
\end{equation}
The solution is unique provided that $F$ additionally satisfies 
\begin{equation}
\label{F_mono}
\int_{\Omega}[ F(\cdot, u, Du, D^2 u) - F(\cdot, v, Dv, D^2 v)](u-v) dx\geq 0 \text{ for all } u, v\in C^2(\overline{\Omega})~\text{with } u= v~\text{on }\p\Omega. 
\end{equation}
 \end{thm}
 \begin{rem}
 It would be very interesting to remove the smallness of $\gamma$ in Theorem \ref{thm_sm}.
 \end{rem}
 Our next result is concerned with radial solutions for $p$-Laplacian right hand side.
 \begin{thm}
 \label{rad_thm}
 Assume that $\Omega= B_1(0)\subset\R^n$ and let $\varphi$ and $\psi$ be constants with $\psi>0$. Let $p\in (1,\infty)$. 
 Let $\beta=p-1$ if $p<2$ and $\beta \in (0,1)$ if $p\geq 2$. Let $f\in\{-1, 1\}$.
Consider the second boundary value problem:
 \begin{equation}
\label{Abreu_rad}
  \left\{ 
  \begin{alignedat}{2}\sum_{i, j=1}^{n}U^{ij}w_{ij}~& =f div (|Du|^{p-2} Du)~&&\text{in} ~\Omega, \\\
 w~&= (\det D^2 u)^{-1}~&&\text{in}~ \Omega,\\\
u ~&=\varphi~&&\text{on}~\p \Omega,\\\
w ~&= \psi~&&\text{on}~\p \Omega.
\end{alignedat}
\right.
\end{equation}
\begin{myindentpar}{0.8cm}
 (i) Let $f=-1$.  Then 
 there is a unique radial, uniformly convex solution $u\in C^{3,\beta}(\overline{\Omega})$
 to (\ref{Abreu_rad}). 
 \\
 (ii) Let $f=1$ and let $p\in (1, n]$. 
 In the case $p=n$, we assume further that $\psi >\frac{1}{n}$.
 Then 
 there is a unique radial, uniformly convex solution $u\in C^{3,\beta}(\overline{\Omega})$
 to (\ref{Abreu_rad}).\\
  (iii) Let $f=1$ and let $p> n$. Suppose that $\psi \geq M(n, p)$ for some sufficiently large constant $M>0$. 
 Then 
 there is a radial, uniformly convex solution $u\in C^{3,\beta}(\overline{\Omega})$
 to (\ref{Abreu_rad}).
\end{myindentpar}
 \end{thm}
 \begin{rem}
 Regarding p-Laplacian right hand side,
 even in the two dimensions, the analysis in \cite{Le18} left open the case $F=-\div (|Du|^{p-2} Du)$ when $p\in (1,2)$. The missing ingredient was the lower bound for $\det D^2 u$ in the a priori estimates. If this is obtained, then one can use the recent result in \cite{Le19}
 to establish the solvability of (\ref{Abreu}); see the proof of Theorem 1.3 in \cite{Le19}.
 \end{rem}
 \vglue 0.2cm
 \begin{rem} 
The size condition on $\psi$ in Theorem \ref{rad_thm} (ii) is optimal. We can see this in two dimensions as follows. If $f\equiv 1$, $n=p=2$ and $0<\psi \leq 1/2$, then there are no
 uniformly convex solutions $u\in C^{4}(\overline{\Omega})$
 to (\ref{Abreu_rad}). Indeed, if such a uniformly convex solution $u$ exists then the first and the last equation of (\ref{Abreu_rad}) implies that $$w(x) = \psi  + \frac{1}{2}(|x|^2-1).$$ However, since $\psi \leq1/2$, there is $x\in\Omega$ such that $w(x)\leq0$, which is a contradiction
to the uniform convexity of $u$ and $w=(\det D^2 u)^{-1}$.
 \end{rem}
 When $n=p=2$, we can remove the symmetry conditions in Theorem \ref{rad_thm}.
\begin{prop}
 \label{thm_C}
 Let $\Omega$ be an open, smooth, bounded and uniformly convex domain in $\R^n$ where $n=2$. Assume $f\geq 0$ and $f\in L^{\infty}(\Omega)$. 
 Assume that $\varphi\in W^{4, q}(\Omega)$, $\psi\in W^{2, q}(\Omega)$ where $q>n$ with 
 \begin{equation}
 \label{size_psi}
 \inf_{x\in \p\Omega}\left(\psi(x)-\frac{\|f\|_{L^{\infty}(\Omega)}}{2}|x|^2\right)>0.
 \end{equation}
 Then there is a uniform convex solution $u\in W^{4,q}(\Omega)$ to the following
 second boundary value problem:
\begin{equation}
\label{Abreu_C}
  \left\{ 
  \begin{alignedat}{2}\sum_{i, j=1}^{n}U^{ij}w_{ij}~& =f \Delta u ~&&\text{in} ~\Omega, \\\
 w~&= (\det D^2 u)^{-1}~&&\text{in}~ \Omega,\\\
u ~&=\varphi~&&\text{on}~\p \Omega,\\\
w ~&= \psi~&&\text{on}~\p \Omega.
\end{alignedat}
\right.
\end{equation}
If $f$ is a nonnegative constant, $\varphi\in C^{\infty}(\overline{\Omega})$, and $\psi\in C^{\infty}(\overline{\Omega})$ then there is a solution $u\in C^{\infty}(\overline{\Omega})$.
 \end{prop}
 
 The key ingredient in the proof of Theorem \ref{thm_sm} is the solvability and uniform estimates in $W^{4, p}(\Omega)$ for $p>n$ of (\ref{Abreu}) when $$F\sim-(\Delta u)^{\frac{1}{n-1}} (\det D^2 u)^{\frac{n-2}{n-1}}$$
 which reduces to $F\sim-\Delta u$ in two dimensions. This result, and its slightly more general version in Proposition \ref{thm_N}, can be of independent interest.

 \begin{prop}
 \label{thm_N}
 Let $\Omega$ be an open, smooth, bounded and uniformly convex domain in $\R^n$.
 Assume that $\varphi\in W^{4, q}(\Omega)$, $\psi\in W^{2, q}(\Omega)$ with $\inf_{\p\Omega}\psi>0$ where $q>n$. Let $k\in\{1,\cdots, n-1\}$.
 Assume that $0\leq f, g\leq 1$. 
 Then there is a uniform convex solution $u\in W^{4,q}(\Omega)$ to the following
 second boundary value problem:
\begin{equation}
\label{Abreu_Newton}
  \left\{ 
  \begin{alignedat}{2}\sum_{i, j=1}^{n}U^{ij}w_{ij}~& =-(\Delta u)^{\frac{1}{n-1}} (\det D^2 u)^{\frac{n-2}{n-1}} f-[S_{k}(D^2 u)]^{\frac{1}{k(n-1)}} (\det D^2 u)^{\frac{n-2}{n-1}} g~&&\text{in} ~\Omega, \\\
 w~&= (\det D^2 u)^{-1}~&&\text{in}~ \Omega,\\\
u ~&=\varphi~&&\text{on}~\p \Omega,\\\
w ~&= \psi~&&\text{on}~\p \Omega.
\end{alignedat}
\right.
\end{equation}
If $f\equiv 1$ and $g\equiv 1$, $\varphi\in C^{4, \beta}(\overline{\Omega})$, and $\psi\in C^{2, \beta}(\overline{\Omega})$ then there is a solution $u\in C^{4,\beta}(\overline{\Omega})$.
 \end{prop}
 In Proposition \ref{thm_N} and what follows, for a symmetric $n\times n$ matrix $A$ with eigenvalues $\lambda_1,\cdots,\lambda_n$, let us denote its elementary symmetric functions $S_k(A)$ where $k=0, 1, \cdots, n$ by
$$S_0(A)=1,~ S_k(A) =\sum_{1\leq i_1<\cdots<i_k\leq n}\lambda_{i_1}\cdots \lambda_{i_k}  (k\geq 1).$$
The rest of the paper is devoted to proving Theorems \ref{thm_sm} and \ref{rad_thm}, and Propositions \ref{thm_C} and  \ref{thm_N}.

\section{Proofs of the main results}
In this section, we prove Theorems \ref{thm_sm} and \ref{rad_thm}, and Propositions \ref{thm_C} and  \ref{thm_N}. As in \cite{Le18}, it suffices to prove appropriate fourth order derivative a priori estimates.

For certain fixed parameters $\beta$ (in Theorem \ref{thm_sm}), $p$ (in Theorem \ref{rad_thm}) and $k, q$  (in Propositions \ref{thm_C} and \ref{thm_N}), we call a positive constant {\it universal} if  it depends only on $n$, $\Omega$, $\psi,\varphi$ and those fixed parameters.
We use $c, C, C_1, C_2,\cdots,$ to denote universal constants and their values may change from line to line.

 \begin{proof}[Proof of Proposition \ref{thm_N}] 
 For simplicity, we denote
 $$F(x)=-(\Delta u(x))^{\frac{1}{n-1}} (\det D^2 u(x))^{\frac{n-2}{n-1}} f(x)-[S_{k}(D^2 u(x))]^{\frac{1}{k(n-1)}} (\det D^2 u(x))^{\frac{n-2}{n-1}} g(x).$$
 We establish a priori estimates for a solution $u\in W^{4, q}(\Omega)$. Since $U^{ij} w_{ij}\leq 0$, by the maximum principle, the function $w$ attains its minimum value on the boundary $\p\Omega$. Thus
 $$w\geq \inf_{\p\Omega}\psi:= C_1>0.$$
 On the other hand, we note that for each $k\in\{1,\cdots, n-1\}$,
 \begin{eqnarray}
 \label{Nineq_00} \Delta u\geq [S_{k}(D^2 u)]^{\frac{1}{k}}
 \end{eqnarray}
 and furthermore,
 \begin{equation}\label{Nineq_01}\trace (U^{ij})=S_{n-1}(D^2 u)\geq (\Delta u)^{\frac{1}{n-1}} (\det D^2 u)^{\frac{n-2}{n-1}}.
 \end{equation}
 Indeed,  (\ref{Nineq_01}) is equivalent to
 $(\det D^2 u) \text{trace}(D^2 u^{-1}) \geq (\Delta u)^{\frac{1}{n-1}} (\det D^2 u)^{\frac{n-2}{n-1}}$,
 or
 \begin{equation}
 \label{Nineq_02}
[\text{Trace}(D^2 u^{-1})]^{n-1} \geq \frac{\Delta u} { \det D^2 u} .
 \end{equation}
 Let $\lambda_1,\cdots, \lambda_n$ be eigenvalues of $D^2 u$. Then (\ref{Nineq_02}) reduces to
 $$ (\sum_{j=1}^n \frac{1}{\lambda_j})^{n-1}\geq \frac{\sum_{i=1}^n \lambda_i}{\prod_{i=1}^n \lambda_i}= \sum_{i=1}^n \prod_{j\neq i}^n\frac{ 1}{\lambda_j}.$$
 This is obvious by the expansion of the left hand side.
 
It follows from (\ref{Nineq_00}) and (\ref{Nineq_01})  and $0\leq f, g\leq 1$ that
$$U^{ij} (w+ |x|^2)_{ij}\geq 0.$$ 
By the maximum principle, the function $w+ |x|^2$ attains its maximum value on the boundary $\p\Omega$. Thus
$$w+|x|^2 \leq \max_{\p\Omega} (\psi+|x|^2)\leq C_2<\infty.$$
Therefore $w\leq C_2$. As a consequence,
$$C_1\leq w\leq C_2.$$
From the second equation of (\ref{Abreu_Newton}), we can find a universal constant $C>0$ such that
\begin{equation}
\label{udet_bound}
C^{-1}\leq\det D^2 u\leq C~\text{in }\Omega.
\end{equation}
 
 By constructing a suitable barrier, we find that $Du$ is universally bounded in $\overline{\Omega}$:
 \begin{equation}
 \label{Dub}
 \|Du\|_{L^{\infty}(\Omega)}\leq C.
 \end{equation}
From $\varphi\in W^{4, q}(\Omega)$ with $q>n$, we have $\varphi\in C^3(\overline{\Omega})$ by the Sobolev embedding theorem. By assumption, $\Omega$ is bounded, smooth and uniformly convex.
From $u=\varphi$ on $\p\Omega$ and (\ref{udet_bound}), we can apply the global $W^{2, 1+\e_0}$ estimates for the Monge-Amp\`ere equation, which
follow from the interior $W^{2, 1+\e_0}$ estimates in De Philippis-Figalli-Savin \cite{DPFS} and Schmidt \cite{Sch} and
the global estimates in 
 Savin \cite{S3} (see also \cite[Theorem 5.3]{Fi}), 
to conclude that
\begin{equation}
\label{D^2uL1}
\|D^2 u\|_{L^{1+\e_0}(\Omega)}\leq C^{\ast}_1
\end{equation}
for some universal constants $\e_0>0$ and $C^{\ast}_1>0$. 

 Thus, from (\ref{D^2uL1}) and (\ref{Nineq_00}), we find that
  $$\|F \|_{L^{(n-1)(1+\e_0)}(\Omega)}\leq C_3$$ for a universal constant $C_3>0$. Note that for all $n\geq 2$ and all $\e_0>0$,
  $$(n-1)(1+\e_0)>n/2.$$
  From $\psi\in W^{2, q}(\Omega)$ with $q>n$, we have $\psi\in C^1(\overline{\Omega})$ by the Sobolev embedding theorem. 
  Now, we apply the global H\"older estimates for the linearized Monge-Amp\`ere equation in \cite[Theorem 1.7 and Lemma 1.5]{LN} to $U^{ij} w_{ij}=F$ in $\Omega$ with boundary value $w=\psi\in C^{1}(\p\Omega)$ on $\p\Omega$ to 
  conclude that $w\in C^{\alpha}(\overline{\Omega})$
  with 
  \begin{equation}
  \label{walpha}
  \|w\|_{C^{\alpha}(\overline{\Omega})}\leq  C\left(\|\psi\|_{C^{1}(\p\Omega)} + \|F\|_{L^{(n-1)(1+\e_0)}(\Omega)}\right)\leq C_4
  \end{equation}
  for universal constants $\alpha\in (0, 1)$ and $C_4>0$. Now, we note that $u$ solves the Monge-Amp\`ere equation
  $$\det D^2 u= w^{-1}$$
  with right hand side being in $C^{\alpha}(\overline{\Omega})$ and boundary value $\varphi\in C^3(\p\Omega)$ on $\p\Omega$.
  Therefore, by the global $C^{2,\alpha}$ estimates for the Monge-Amp\`ere equation \cite{TW08, S}, we have $u\in C^{2,\alpha}(\overline{\Omega})$
  with universal estimates
  \begin{equation}
  \label{u2alpha}
   \|u\|_{C^{2,\alpha}(\overline{\Omega})}\leq C_5~\text{and } C_5^{-1} I_n\leq D^2 u\leq C_5 I_n.
   \end{equation}
    Here and throughout, we use $I_n$ to denote the $n\times n$ identity matrix.
   As a consequence, the second order operator $U^{ij} \p_{ij}$ is uniformly elliptic with H\"older continuous coefficients. 
   Now, we observe from the definition of $F$ and (\ref{u2alpha}) that
\begin{equation}
\label{C5_eq}
\|F \|_{L^{\infty}(\Omega)}\leq C_6.
\end{equation}
Thus, from the equation $U^{ij} w_{ij}=F$ with boundary value $w=\psi$ where $\psi\in W^{2, q}(\Omega)$, we conclude that 
 $w\in W^{2, q}(\Omega)$ and therefore $u\in W^{4,q}(\Omega)$ with universal estimate
 \begin{equation*}
   \|u\|_{W^{4,q}(\Omega)}\leq C_7.
   \end{equation*}
   
It remains to consider the case $f\equiv 1$ and $g\equiv 1$, $\varphi\in C^{4, \beta}(\overline{\Omega})$, and $\psi\in C^{2, \beta}(\overline{\Omega})$. 
In this case,  we need to establish a priori estimates for $u\in C^{4,\beta}(\overline{\Omega})$. As above, instead of (\ref{C5_eq}), we have
\begin{equation}
\label{C7_eq}
\|F \|_{C^{\frac{\alpha}{n-1}}(\overline{\Omega})}\leq C_7.
\end{equation}
Thus, from the equation $U^{ij} w_{ij}=F$ with boundary value $w=\psi$ where $\psi\in C^{2, \beta}(\overline{\Omega})$, we conclude that 
 $w\in C^{2, \gamma}(\overline\Omega)$ where
 $\gamma:= \min\{\frac{\alpha}{n-1},\beta\}$
  and therefore $u\in C^{4,\gamma}(\overline{\Omega})$ with the universal estimate
 $
   \|u\|_{C^{4,\gamma}(\overline{\Omega})}\leq C_8.
   $
  With this estimate, we can improve (\ref{C7_eq}) to  
  \begin{equation}
\label{C9_eq}
\|F \|_{C^{\beta}(\overline{\Omega})}\leq C_9.
\end{equation}
As above, we find that 
$u\in C^{4,\beta}(\overline{\Omega})$ with the universal estimate
 $
   \|u\|_{C^{4,\beta}(\overline{\Omega})}\leq C_{10}.
   $
\end{proof}
\begin{proof}[Proof of Theorem \ref{thm_sm}]
Without loss of generality, we can assume that $\inf_{\p\Omega}\psi =1.$
We consider the following second boundary value problem for a uniformly convex function $u$:
\begin{equation}
\label{Abreu_Newton2}
  \left\{ 
  \begin{alignedat}{2}U^{ij}w_{ij}~& =-(\Delta u)^{\frac{1}{n-1}} (\det D^2 u)^{\frac{n-2}{n-1}} f_\gamma(\cdot, u, Du, D^2 u)~&&\text{in} ~\Omega, \\\
 w~&= (\det D^2 u)^{-1}~&&\text{in}~ \Omega,\\\
u ~&=\varphi~&&\text{on}~\p \Omega,\\\
w ~&= \psi~&&\text{on}~\p \Omega.
\end{alignedat}
\right.
\end{equation}
for some $\gamma\in (0, 1)$ to be chosen later,
where
$$f_\gamma(\cdot, u, Du, D^2 u)=\min\{ \frac{-\gamma F(\cdot, u, Du, D^2 u)}{(\Delta u)^{\frac{1}{n-1}} (\det D^2 u)^{\frac{n-2}{n-1}}}, 1\}.$$
By our assumption (ii) on $F$, when $u$ is a  $C^2$ convex function, we have $0\leq f_\gamma\leq 1$.
By Proposition \ref{thm_N} (with $g\equiv 0$), (\ref{Abreu_Newton2}) has a solution $u\in W^{4, q}(\Omega)$ for all $q<\infty$. Thus, the first equation of (\ref{Abreu_Newton2})  holds pointwise a.e.

As in the proof of Proposition \ref{thm_N} (see (\ref{u2alpha})), we have the following a priori estimates
\begin{equation}
\label{u2_rec}
\| u \|_{C^{2,\beta}(\overline{\Omega})}\leq C_1~\text{and } C_1^{-1} I_n \leq D^2 u\leq C_1 I_n
\end{equation}
for some $C_1>0$ depending only on $\beta, \varphi, \psi, n$ and $\Omega$. 
Hence, using the assumption (i) on $F$, we find that
$$ \frac{-\gamma F(\cdot, u, Du, D^2 u)}{(\Delta u)^{\frac{1}{n-1}} (\det D^2 u)^{\frac{n-2}{n-1}}}<\frac{1}{2}$$
if $\gamma>0$ is small, depending only on $\beta, \varphi, \psi, n, F$ and $\Omega$.

Thus, if $\gamma>0$ is small, depending only on $\beta, \varphi, \psi, n, F$ and $\Omega$, then 
$$f_\gamma = \min\{ \frac{-\gamma F(\cdot, u, Du, D^2 u)}{(\Delta u)^{\frac{1}{n-1}} (\det D^2 u)^{\frac{n-2}{n-1}}}, 1\}=\frac{-\gamma F(\cdot, u, Du, D^2 u)}{(\Delta u)^{\frac{1}{n-1}} (\det D^2 u)^{\frac{n-2}{n-1}}}$$
in $\Omega$ and hence the first equation of (\ref{Abreu_Newton2})  becomes
$$U^{ij} w_{ij}=\gamma F(\cdot, u, Du, D^2 u).$$
Using this equation together with (\ref{u2_rec}) and $\varphi\in C^{4,\beta}(\overline{\Omega})$ and $\psi \in C^{2,\beta}(\overline{\Omega})$, we easily conclude $u\in  C^{4,\beta}(\overline{\Omega})$.
Thus, there is a uniform convex solution $u\in C^{4,\beta}(\overline{\Omega})$ to (\ref{Abreu_small}).

Assume now $F$ additionally satisfies (\ref{F_mono}). Then arguing as in the proof of \cite[Lemma 4.5]{Le18} replacing $f_\delta$ there by $\gamma F$, we obtain the uniqueness of $C^{4,\beta}(\overline{\Omega})$ solution to (\ref{Abreu_small}).
\end{proof}
\begin{rem}
Clearly, Theorem \ref{thm_sm} and its proof apply to dimensions $n=2$.
\end{rem}

\begin{proof}[Proof of Proposition \ref{thm_C}]  We stablish a priori estimates for a solution $u\in W^{4, q}(\Omega)$ to (\ref{Abreu_C}).
As in the proof of Proposition \ref{thm_N}, it suffices to obtain the lower and upper bounds on $\det D^2 u$.

 Observe that $$U^{ij} w_{ij} = f\Delta u\geq 0.$$ By the maximum principle, the function $w$ attains its maximum value on the boundary $\p\Omega$. Thus
 $$w\leq \sup_{\p\Omega}\psi<\infty.$$
 By the second equation of (\ref{Abreu_C}), this gives a bound from below for $\det D^2 u$:
 $$\det D^2 u\geq C^{-1}.$$
 On the other hand, we have
 $$\sum_{i, j=1}^2U^{ij} (w-\frac{\|f\|_{L^{\infty}(\Omega)}}{2}|x|^2)_{ij} = (f-\|f\|_{L^{\infty}(\Omega)})\Delta u\leq 0.$$
 By the maximum principle, the function $w-\frac{\|f\|_{L^{\infty}(\Omega)}}{2}|x|^2$ attains its minimum value on the boundary $\p\Omega$. Thus, using (\ref{size_psi}), we find that
 $$w-\frac{\|f\|_{L^{\infty}(\Omega)}}{2}|x|^2\geq \inf_{x\in\p \Omega}\left(w(x)-\frac{\|f\|_{L^{\infty}(\Omega)}}{2}|x|^2\right)>0.$$
This gives a positive lower bound for $w$, that is, 
$w\geq C^{-1}>0.$
Using the second equation of (\ref{Abreu_C}), we obtain a bound from above for $\det D^2 u$:
 $$\det D^2 u\leq C.$$
\end{proof}

 \begin{proof}[Proof of Theorem \ref{rad_thm}] 
 Recall that $f\in \{-1, 1\}$ and $\beta=p-1$ if $1<p<2$ and $\beta \in (0,1)$ if $p\geq 2$.
 
We first observe the following reduction of smoothness without any symmetry assumptions. Suppose that one has a uniformly convex solution $u\in C^2(\overline{\Omega})$ to (\ref{Abreu_rad}) with positive lower and upper bounds
on $\det D^2 u$:
\begin{equation}
\label{2side}
C^{-1}\leq \det D^2 u\leq C
\end{equation} 
for some $C>0$
and such that $w\in C^{\beta}(\overline{\Omega})$, then $u\in C^{3,\beta}(\overline{\Omega})$. Indeed,  using (\ref{2side}) together with the global $C^{2,\alpha}$ estimates \cite{TW08, S} for the Monge-Amp\`ere equation $\det D^2 u=w^{-1}$ with boundary data $\varphi \in C^{3,1}(\overline{\Omega})$ and right hand side $w^{-1}\in C^{\beta}(\overline{\Omega})$,  we have $u\in C^{2,\beta}(\overline{\Omega})$
 with estimates
  \begin{equation}
  \label{u2alpha_rad}
   \|u\|_{C^{2,\beta}(\overline{\Omega})}\leq C_1~\text{and } C_1^{-1} I_n\leq D^2 u\leq C_1 I_n.
   \end{equation}
     As a consequence, the second order operator $U^{ij} \p_{ij}$ is uniformly elliptic with H\"older continuous coefficients with exponent $\beta\in (0,1)$. 
   Note that $|Du|^{p-2} Du$ is H\"older continuous with exponent $\beta$. 
   Using the first equation of (\ref{Abreu_rad}), we see that the $C^{1,\beta}(\overline{\Omega})$ estimates for $w$ follows from \cite[Theorem 8.33]{GT}. Hence, we have the 
   $C^{3,\beta}(\overline{\Omega})$ estimates for $u$. 

 Now, we look for radial, uniformly convex solutions $u\in C^{2}(\overline{\Omega})$ to (\ref{Abreu_rad}).
 Assume that the convex function $u$ is of the form
$$u(x)= v(r)$$ where
$$v: [0, \infty)\rightarrow \R~\text{and }r=|x|.$$
Let us denote
$$^{'}=\frac{d}{dr}\quad\text{and } g(r):= v'(r).$$
The requirement that $u\in C^{2}(\overline{\Omega})$ forces
$$g(0)=v'(0)=0.$$
The next reduction in the proof of our theorem is the following claim.\\
{\bf Claim.} The existence of radial, uniformly convex solutions $u\in C^{2}(\overline{\Omega})$ to (\ref{Abreu_rad}) with  positive lower and upper bounds
on $\det D^2 u$ and a Holder continuous $w$ is equivalent to finding $g(1)>0$ satisfying the integral equation
\begin{equation}
\label{ccond}
\int_{0}^{g(1)} e ^{\frac{f}{p} s^p} s^{n-1} ds = \frac{1}{n\psi}\left(1+ f\int_0^{g(1)} e ^{\frac{f}{p} s^p} s^{p-1} ds\right).
\end{equation}

To prove the claim,
we compute
$$\det D^2 u = v^{''}(\frac{v'}{r})^{n-1},~
w= (\det D^2 u)^{-1}=\frac{1}{v^{''}} (\frac{r}{v'})^{n-1} \equiv W(r).
$$
Since $D^2 u$ and $(D^2 u)^{-1}$ are similar to 
$\text{diag}~ (v^{''}, \frac{v'}{r}, \cdots, \frac{v'}{r})$
and 
$\text{diag}~(\frac{1}{v^{''}}, \frac{r}{v^{'}}, \cdots, \frac{r}{v^{'}}),$ we can
compute
\begin{eqnarray*}U^{ij}w_{ij}= \frac{v^{''} (v^{'})^{n-1}}{r^{n-1}} \left(\frac{W^{''}}{v^{''}} + (n-1) \frac{W^{'}}{v^{'}}\right)
= \frac{[W^{'} (v^{'})^{n-1}]^{'}}{r^{n-1}}.
\end{eqnarray*}
Note that $v^{''}$ and $v'$ are all nonnegative.  Therefore,
\begin{equation}
\label{v_mono}
0\leq v'(r) \leq v'(1)~\text{for all } 0\leq r\leq 1.
\end{equation}
On the other hand, we have
$$\div (|Du|^{p-2} Du) = (p-1) (v')^{p-2}v^{''} + \frac{n-1}{r}(v')^{p-1}=\frac{[(v')^{p-1} r^{n-1}]^{'}}{r^{n-1}}.$$
The first equation of (\ref{Abreu_rad}) gives
$$ \frac{[W^{'} (v^{'})^{n-1}]^{'}}{r^{n-1}}= f\frac{[(v')^{p-1} r^{n-1}]^{'}}{r^{n-1}}$$
which implies that, for some constant $C$
$$W^{'} (v^{'})^{n-1} = f (v')^{p-1} r^{n-1} + C.$$
Since $v'(0)=0$, we find that $C=0$. Thus
$$W^{'}= f (v')^{p-1} (\frac{r}{v'})^{n-1}= f (v')^{p-1} v^{''} W.$$
It follows that
$$[\log W]^{'} = [\frac{f}{p}(v')^p]^{'}$$
and hence, recalling $W(1)=\psi$,
\begin{equation}
\label{W_eq}
\log W(r) = \log W(1) + \frac{f}{p}\left[(v'(r))^p - (v'(1))^p\right]=  \log \psi +  \frac{f}{p}\left[(v'(r))^p - (v'(1))^p\right].
\end{equation}
Therefore, in terms of $g=v^{'}$, we have after exponentiation,
\begin{equation}
\label{greq0}
e^{\frac{f}{p} [g(r)]^p} [g(r)]^{n-1} g'(r) = \frac{1}{\psi} e^{\frac{f}{p} [g(1)]^p} r^{n-1},
\end{equation}
which is equivalent to
\begin{equation}
\label{greq}
\int_{0}^{g(r)} e ^{\frac{f}{p} s^p} s^{n-1} ds = \frac{1}{n\psi} e^{\frac{f}{p} [g(1)]^p} r^{n}.
\end{equation}
Clearly, (\ref{greq}) leads to a solution to (\ref{Abreu_rad}) in terms of $g(1), n, p$ and $\psi$
 provided $g(1)>0$ satisfies the compatibility condition at $r=1$:
\begin{equation}\label{ccond2}
\int_{0}^{g(1)} e ^{\frac{f}{p} s^p} s^{n-1} ds = \frac{1}{n\psi} e^{\frac{f}{p} [g(1)]^p}.
\end{equation}
Because
$$e^{\frac{f}{p} [g(1)]^p}=1+ f\int_0^{g(1)} e ^{\frac{f}{p} s^p} s^{p-1} ds,$$
the compatibility condition (\ref{ccond2}) can be rewritten as in (\ref{ccond}).

Assume that $g(1)= v^{'}(1)>0$ has already been found, in terms of $n, p$ and $\psi$. We now establish positive lower and upper bounds
on $\det D^2 u$ and that $w\in C^{\beta}(\overline{\Omega})$. Indeed, from $0\leq g(r) \leq g(1)$, we can easily estimate
$$ e ^{\frac{-1}{p} [g(1)]^p} \frac{[g(r)]^n}{n}\leq  \int_{0}^{g(r)} e ^{\frac{f}{p} s^p} s^{n-1} ds  \leq e ^{\frac{1}{p} [g(1)]^p} 
\frac{[g(r)]^n}{n}.$$
Hence (\ref{greq}) gives $$C^{-1} r\leq g(r)\leq Cr$$ for some $C$ depends only on $g(1)>0$, $n, p$ and $\psi$. Thus, from (\ref{greq0}), we find that $v^{''}$ and $\frac{v^{'}(r)}{r}$ are bounded from below and above by positive constants. Therefore, we have positive lower and upper bounds
on $\det D^2 u=  v^{''}(\frac{v'}{r})^{n-1}$. Moreover, $v'(r) = |Du(x)|\in C^{\alpha}(\overline{\Omega})$
for all $\alpha\in (0, 1)$. 
Using (\ref{W_eq}), 
we also find that $W$, and hence $w$, is in $C^{\alpha}(\overline{\Omega})$. In particular, $w\in C^{\beta}(\overline{\Omega})$.
   
We have reduced our theorem to the existence and uniqueness of $g(1)>0$ solving (\ref{ccond}) which we now address.\\
(i) Recall that $f=-1$. Note that (\ref{ccond}) becomes
\begin{equation*}
\int_{0}^{g(1)} e ^{\frac{-1}{p} s^p} s^{n-1} ds = \frac{1}{n\psi}\left(1-\int_0^{g(1)} e ^{\frac{-1}{p} s^p} s^{p-1} ds\right).
\end{equation*}
Clearly, there is a unique $g(1)>0$ solving the above integral equation. Hence, there is a unique radial, uniformly convex solution $u\in C^{3,\beta}(\overline{\Omega})$
 to (\ref{Abreu_rad}). \\
(ii) Recall that $f=1$ and $p\in (1, n]$. Note that (\ref{ccond}) becomes
\begin{equation}
\label{HIeq}
H(g(1)= I(g(1))
\end{equation} where
\begin{equation}
\label{HIdef}
H(t):=\int_{0}^{t} e ^{\frac{1}{p} s^p} s^{n-1} ds \quad\text{and } I(t):= \frac{1}{n\psi}\left(1+\int_0^{t} e ^{\frac{1}{p} s^p} s^{p-1} ds\right)\equiv \frac{1}{n\psi} e^{\frac{t^p}{p}}.
\end{equation}
Consider first the case $p=n$. Then 
$$H(t) = e^{\frac{t^n}{n}}-1\quad \text{and } I(t)= \frac{1}{n\psi} e^{\frac{t^n}{n}}.$$ 
Therefore, from (\ref{HIeq}) we find an explicit formula for $g(1)$ from the equation
$$e^{\frac{1}{n} [g(1)]^n}= \frac{n\psi}{n\psi-1},$$
showing that existence and uniqueness of a solution $g(1)>0$ to (\ref{ccond}) when 
 $\psi >\frac{1}{n}$. As a result, there is a unique radial, uniformly convex solution $u\in C^{3,\beta}(\overline{\Omega})$
 to (\ref{Abreu_rad}). Moreover, $\psi>\frac{1}{n}$ is also the optimal condition for the existence of a radial solution to (\ref{Abreu_rad}).

Now we consider the case $p\in (1, n)$ and  $\psi >0$. We show that (\ref{HIeq}) has a unique solution $g(1)>0$ and hence there is a unique radial, uniformly convex solution $u\in C^{3,\beta}(\overline{\Omega})$
 to (\ref{Abreu_rad}). Indeed, since $1<p<n$, the integrand of $H(t)$ grows faster than that of $I(t)$. Since $H(0)=0< I(0)=\frac{1}{n\psi}$, the function $H(t)$ will cross $I(t)$ for the first time from below at some point $t_0>0$. Thus $g(1)=t_0>0$
 is a solution of (\ref{HIeq}). To show the uniqueness of $g(1)$, we show that if $t>t_0$ then $H(t)> I (t)$. 
Indeed, using the definition of $t_0$, we find that $H'(t_0)\geq I'(t_0)$. This means that
$$e ^{\frac{1}{p} t_0^p} t_0^{n-1} \geq \frac{1}{n\psi}e ^{\frac{1}{p} t_0^p} t_0^{p-1},$$
or, equivalently,
$$t_0^{n-p}\geq \frac{1}{n\psi}.$$
Thus, if $s>t_0$ then
$s^{n-p}> \frac{1}{n\psi}$, that is,
$$e ^{\frac{1}{p} s^p} s^{n-1} > \frac{1}{n\psi}e ^{\frac{1}{p} s^p} s^{p-1},$$
and hence, for any $t>t_0$, we have
$$H(t) = H(t_0) +\int_{t_0}^t e ^{\frac{1}{p} s^p} s^{n-1} ds > I(t_0) + \frac{1}{n\psi}\int_{t_0}^t e ^{\frac{1}{p} s^p} s^{p-1} ds= I(t).$$

(iii) Recall that $f=1$ and $p>n$. Assume that $\psi\geq M(n, p):= 1+ \frac{e^{1/p}}{n}\left(\int_{0}^1 e ^{\frac{1}{p} s^p} s^{n-1} ds\right)^{-1}$. Then, there is  a solution  $g(1)>0$ to (\ref{HIeq}) where $H$ and $I$ are defined as in (\ref{HIdef}). Indeed, in this case, we have
$1>  \frac{e^{1/p}}{n\psi} [H(1)]^{-1}= \frac{I(1)}{H(1)}$. Therefore $I(1)<H(1)$ while 
$I(0)>H(0)$. Thus, (\ref{HIeq}) has a solution $g(1)\in (0, 1)$. Consequently,  there is a radial, uniformly convex solution $u\in C^{3,\beta}(\overline{\Omega})$
 to (\ref{Abreu_rad}). 
 \end{proof}
\begin{rem}
When $p>n$, radial solutions in Theorem \ref{rad_thm} (iii) are not unique in general. This corresponds to multiple crossings of $H$ and $I$ defined in (\ref{HIdef}). For example, this is in fact the case of $n=2$, $p=4$ and $\psi=1$. We can plot the graphs of 
$H$ and $I$ using Maple to find that, on $[0, 2]$, they cross twice at $t_1\in (1, 6/5)$ and $t_2 \in (3/2, 2)$. 
\end{rem}
{\bf Acknowledgement.} The author would like to thank Connor Mooney for critical comments on a previous version of this paper. The author also thanks the anonymous referee for his/her crucial comments and suggestions that help strengthen and simplify the proof of Theorem \ref{rad_thm}.

\end{document}